\documentclass{amsart}
\usepackage{amsmath}
\usepackage{amssymb}
\usepackage{amsthm}
\usepackage{mathrsfs}

\usepackage{graphicx}

\newtheorem{theorem}{Theorem}[section]
\newtheorem{lemma}[theorem]{Lemma}

\theoremstyle{definition}
\newtheorem{definition}[theorem]{Definition}

\newtheorem{corollary}[theorem]{Corollary}

\theoremstyle{remark}
\newtheorem{remark}[theorem]{Remark}

\numberwithin{equation}{section}

\newcommand{\RR} {\mathbb R}

\newcommand{\BBB} {\mathcal B}

\newcommand{\MMM} {\mathcal M}

\newcommand{\ra} {\rightarrow}

\begin{document}

\title{On structure space of the ring $B_1(X)$}

\author{A. Deb Ray}
\address{Department of Pure Mathematics, University of Calcutta, 35, Ballygunge Circular Road, Kolkata - 700019, INDIA} 
\email{debrayatasi@gmail.com}
\author{Atanu Mondal}
\address{Department of Commerce (E), St. Xavier's college, 30, Mother Teresa sarani, Kolkata - 700016, INDIA}
\email{atanu@sxccal.edu}

\begin{abstract}
In this article, we continue our study of the ring of Baire one functions on a topological space $(X,\tau)$, denoted by $B_1(X)$ and extend the well known M. H. Stones's theorem from $C(X)$ to $B_1(X)$. Introducing the structure space of $B_1(X)$, an analogue of Gelfand Kolmogoroff theorem is established. It is observed that $(X,\tau)$ may not be embedded inside the structure space of $B_1(X)$. This observation inspired us to introduce a weaker form of embedding and show that in case $X$ is a $T_4$ space, $X$ is weakly embedded as a dense subspace, in the structure space of $B_1(X)$. It is further established that the ring $B_1^{*}(X)$ of all bounded Baire one functions is a C-type ring and also, the structure space of $B_1^{*}(X)$ is homeomorphic to the structure space of $B_1(X)$. Introducing a finer topology $\sigma$ than the original $T_4$ topology $\tau$ on $X$, it is proved that $B_1(X)$ contains free (maximal) ideals if $\sigma$ is strictly finer than $\tau$. It is also proved that $\tau = \sigma$ if and only if $B_1(X) = C(X)$. Moreover, in the class of all perfectly normal $T_1$ spaces, $B_1(X) = C(X)$ is equivalent to the discreteness of the space $X$.
\end{abstract}

\maketitle

\section{Introduction and Prerequisites}
\noindent The collection $B_1(X)$, of all real valued Baire one functions defined on a topological space $X$ forms a commutative lattice ordered ring with unity. Initiating the study of $B_1(X)$ in \cite{AA} we have established a duality between the ideals of $B_1(X)$ and $Z_B$-filters (an analogue of $Z$-filters) on $X$ in a subsequent paper \cite{AA2}. \\
\noindent In case of the rings of continuous functions, M. H. Stone's theorem states that, for every topological space $X$ there exists a Tychonoff space $Y$ such that $C(X) \cong C(Y)$, which is extremely important and useful. Since $B_1(X)$ is a ring that contains $C(X)$ as a subring, it is natural to ask whether it is possible to extend the celebrated M. H. Stone's theorem \cite{GJ} in this bigger ring. In this paper, we begin our study of $B_1(X)$ by addressing this question and answer it in affirmative. Therefore, in view of this result, it would be enough to deal with Tychonoff spaces as long as the study of the ring structure of $B_1(X)$ is concerned.  \\\\  
\noindent The collection of all maximal ideals of $C(X)$, denoted by $\MMM(C(X))$, equipped with hull-kernel topology is known as the structure space of the ring $C(X)$. It is also very well known that the structure space of $C(X)$ is homeomorphic to the collection of all $Z$-ultrafilters on $X$ with Stone topology \cite{GJ}. In section 2, defining the structure space of $B_1(X)$ in a similar manner, we could establish an analogue of this result in the context of the ring $B_1(X)$. The importance of the structure space of $C(X)$ for a Tychonoff space $X$ lies in the fact that a copy of $X$ is densely embedded in it, i.e., a Tychonoff space $X$ is embedded in the space $\MMM(C(X))$ with hull-kernel topology. Moreover, the structure space $\MMM(C(X))$ becomes the $Stone$-$\check{C}ech$ compactification of $X$. But in case of $\MMM(B_1(X))$, it may not happen the same way. The space $(X, \tau)$ may not be embedded in the structure space $\MMM(B_1(X))$ of the ring $B_1(X)$. We have shown that in case $X$ is a $T_4$-space, a weaker form of embedding from $X$ into the structure space $\MMM(B_1(X))$ exists. Being inspired by this fact, we have introduced another topology $\sigma$ on $X$, generally finer than $\tau$, such that $(X, \sigma)$ is densely embedded inside $\MMM(B_1(X))$. This result leads to several important conclusions. It is proved that the ring $B_1^*(X)$ of bounded Baire one functions is a C-type ring and $\MMM(B_1(X)) \cong \MMM(B_1^*(X))$. Finally, for compact $T_2$ spaces, $\tau \subsetneq \sigma$ ensures the existence of free maximal ideals in $B_1(X)$ and within the class of perfectly normal $T_1$ spaces at least, $B_1(X) = C(X)$ is equivalent to the space to be discrete. \\\\
\noindent In what follows, we write $X$, $Y$ etc. to denote topological spaces without mentioning their topologies (unless required) explicitly.  
\section{Extension of M. H. Stone's Theorem}  
\noindent As proposed in the introduction, we construct an isomorphism $B_1(Y) \rightarrow B_1(X)$ using the existing isomorphism $C(Y) \rightarrow C(X)$, where $Y$ is the Tychonoff space constructed suitably from a given topological space $X$. It is also not very hard to observe that such an isomorphism is a lattice isomorphism. 
\begin{theorem}\label{MHS}
For each topological space $X$, there exists a Tychonoff space $Y$ such that $B_1(X)$ is isomorphic to $B_1(Y)$ and $B_1^*(X)$ is isomorphic to $B_1^*(Y)$ under the same (restriction) map. 
\end{theorem}
\begin{proof}
Define a binary relation $``\sim"$ on $X$ by $x \sim y$ if and only if $f(x)=f(y)$, for all $f \in C(X)$. $\sim$ is an equivalence relation on $X$. Let $Y= X /\sim$ $\equiv \{[x]: x \in X\}$, where $[x]$ denotes the equivalence class of $x\in X$.\\
Define $\tau :X \rightarrow Y$ by $\tau(x)=[x]$, for all $x \in X$.\\
\noindent For each $f \in C(X)$, let $g_f \in C(Y)$ be defined by the rule $g_f([x])=f(x)$, for all $[x] \in Y$. Certainly $g_f$ is well defined and $g_f \circ \tau = f$. \\
\noindent Consider $C'=\{g_f: f \in C(X)\}$ and equip $Y$ with the weak topology induced by the family $C'$. Then $Y$ becomes a completely regular space \cite{GJ}. Also, $g_f \circ \tau $ is continuous for all $g_f \in C'$. Hence $\tau$ is continuous. \\
\noindent If $g \in C(Y)$ then $g\circ \tau \in C(X)$ and hence $g \circ \tau = f$, for some $f \in C(X)$. So, $\forall$ $[x] \in Y$, $g([x])=g(\tau(x))=f(x)=(g_f\circ \tau )(x)=g_f([x])$. i.e., $g = g_f$ and consequently, $C' = C(Y)$.\\ 
\noindent Let $[x] \neq [y]$ in $Y$. Then there exists $f \in C(X)$ such that $f(x) \neq f(y)$. So, $g_f[x] \neq g_f[y]$. This proves that $Y$ is Hausdorff and hence, a Tychonoff space.\\
\noindent Let $h \in B_1(Y)$ be any Baire one function on $Y$. There exists a sequence of continuous functions $\{h_n\} \subset C(Y)$ such that, $\{h_n\}$ converges pointwise to $h$, i.e., $\lim\limits_{n\to\infty}h_n(x)=h(x)$, for all $x \in X$. Clearly, $h_n \circ \tau \in C(X)$, $\forall n \in \mathbb{N}$ and also $\lim\limits_{n\to\infty}(h_n\circ \tau)(x)$ exists for all $x \in X$.\\
\noindent Define $\widehat{\psi}(h): X \rightarrow \mathbb{R}$ by $\widehat{\psi}(h)(x)=\lim\limits_{n\to\infty}(h_n\circ \tau)(x)$, for each $h \in B_1(Y)$. Then $\widehat{\psi}(h) \in B_1(X)$. Finally, define $\widehat{\psi}: B_1(Y) \rightarrow B_1(X)$ by $h \mapsto \widehat{\psi}(h)$. It is easy to check that $\widehat{\psi}$ is an isomorphism and in view of a result proved in \cite{AA}, the restriction of $\widehat{\psi}$ on $B_1^*(Y)$ to $B_1^*(X)$ is also an isomorphism.
\end{proof}
\noindent In \cite{AA}, we have established that every ring homomorphism $B_1(Y) \ra B_1(X)$ is a lattice homomorphism. As a consequence, we get
\begin{corollary}
The isomorphism $\psi : B_1(Y) \ra B_1(X)$ is a lattice isomorphism.
\end{corollary}
\noindent Theorem~\ref{MHS} ensures that it is enough to study the ring of Baire one functions defined on any Tychonoff space, instead of any arbitrary topological space. Therefore, in the rest of this paper, by a topological space we always mean a Tychonoff space, unless stated otherwise.
\section{The structure space of $B_1(X)$}
\noindent Let $X$ be a Tychonoff space. Consider $\MMM(B_1(X))$ as the collection of all maximal ideals of the ring $B_1(X)$. It is easy to observe that for each $f \in B_1(X)$, if $\widehat{\mathscr M_f} = \{\widehat{M} \in \MMM(B_1(X)) \ : \ f \in \widehat{M}\}$ then the collection $\{\widehat{\mathscr M_f} \ : \ f \in B_1(X)\}$ forms a \textbf{base for closed sets} for some topology $\zeta$ on $\MMM(B_1(X))$. This topological space $(\MMM(B_1(X)), \zeta)$ is called the \textbf{structure space} of $B_1(X)$ and the topology is known as the \textbf{hull-kernel topology}. It is well known that the structure space of any commutative ring with unity is always compact. Moreover, the structure space is Hausdorff if the ring is Gelfand (i.e., a ring where every prime ideal can be extended to a unique maximal ideal). Therefore, 
\begin{itemize}
    \item $\MMM(B_1(X))$ is \textbf{compact}. 
    \item $\MMM(B_1(X))$ is \textbf{Hausdorff}, since $B_1(X)$ is a Gelfand ring~\cite{AA2}.
\end{itemize} 
\noindent In \cite{AA2}, we have introduced $Z_B$-filter, $Z_B$-ultrafilter and studied their interplay with ideals and maximal ideals of $B_1(X)$. It has been observed that a bijective correspondence exists between the collection of all maximal ideals of $B_1(X)$ $\left(\equiv \MMM(B_1(X))\right)$ and the collection of all $Z_B$-ultrafilters on $X$. We now show that the structure space of $B_1(X)$, i.e., $\MMM(B_1(X))$ with \textbf{hull-kernel topology} is homeomorphic to the set of all $Z_B$-ultrafilters on X with Stone-topology.\\\\
\noindent We know that for each $p \in X$, $\mathscr U_p=\{Z \in Z(B_1(X)): p \in Z\}$ is a $Z_B$-ultrafilter on X. In fact $Z[\widehat{M_p}]=\mathscr U_p$. So, we can use the set $X$ as the index set for all $Z_B$-ultrafilters on $X$ of the form $\mathscr U_p$. We enlarge the set $X$ to a bigger set $\widetilde{X}$, which serves as an index set for the family of all $Z_B$-ultrafilters on $X$. For each $p \in \widetilde{X}$, let the
corresponding $Z_B$-ultrafilter be denoted by $\mathscr U^p$ and whenever $p \in X$, we take $\mathscr U^p = \mathscr U_p = \{Z \in Z(B_1(X)): p \in Z\}$. So, $\{{\mathscr U^p: p \in \widetilde{X}}\}$ is the set of all $Z_B$-ultrafilters on $X$.\\
For each $Z \in Z(B_1(X))$, let $\overline{Z}= \{p \in \widetilde{X}: Z \in \mathscr U^p \}$. If $p \in Z$ then $Z \in \mathscr U^p = \mathscr U_p$ and hence, $p \in \overline{Z}$. i.e., $Z \subseteq \overline{Z}$. Also $\overline{X}= \widetilde{X}$. The collection $\mathscr B=\{\overline{Z}: Z \in Z(B_1(X))\}$ forms a base for closed sets for some topology on $\widetilde{X}$, as
\begin{enumerate}
	\item $\overline{\emptyset} = \{p \in \widetilde{X}:\emptyset \in \mathscr U^p\}= \emptyset \implies \emptyset \in \mathscr B$ .
	\item For $Z_1$ and $Z_2$ $\in Z(B_1(X))$, $\overline{Z_1 \cup Z_2}= \overline{Z_1} \cup \overline{Z_2}\implies$ union of two sets in $\mathscr B$ belongs to $\mathscr B$. 
\end{enumerate}
This topology is known as \textbf{Stone-topology}. We simply write $\widetilde{X}$ to mean the space $\widetilde{X}$ with Stone-topology. It is easy to check that for any $Z_1$ and $Z_2$ with $Z_1 \subseteq Z_2$ implies $\overline{Z_1} \subseteq \overline{Z_2}$ and also, $\overline{Z} \cap X= Z$. As a consequence, we get the following result:
\begin{theorem}.
	For any $Z \in Z(B_1(X))$, $\overline{Z}= cl_{\widetilde{X}}Z$. \\
	In particular,  $cl_{\widetilde{X}}X=\widetilde{X}$.
\end{theorem}
\begin{proof}
	Straightforward and hence omitted.
\end{proof}
\noindent For each maximal ideal $\widehat{M}$ in $B_1(X)$, $Z[\widehat{M}]$ is a unique $Z_B$-ultrafilter on $X$. Hence, $Z[\widehat{M}]= \mathscr U^p$, for some unique $p \in \widetilde{X}$. Therefore, define a map $\Phi : \MMM(B_1(X)) \rightarrow  \widetilde{X}$ by $\Phi(\widehat{M})=p$, whenever $Z[\widehat{M}]= \mathscr U^p$.
\begin{theorem}
The structure space $\MMM(B_1(X))$ of the ring $B_1(X)$ is homeomorphic to $\widetilde{X}$ with Stone-topology.
\end{theorem}
\begin{proof}
The map $\Phi : \MMM(B_1(X)) \rightarrow  \widetilde{X}$ defined by $\Phi(\widehat{M})=p$, whenever $Z[\widehat{M}]= \mathscr U^p$ is a bijection between $\MMM(B_1(X))$ and $\widetilde{X}$, because it was proved in \cite{AA2} that $\widehat{M},\widehat{N} \in \MMM(B_1(X))$ with $\widehat{M} \neq \widehat{N}$ implies $Z[\widehat{M}] \neq Z[\widehat{N}]$. Also, the collection $\{\widehat{\mathscr M_f}: f \in B_1(X)\}$ is a base for closed sets for the structure space of $B_1(X)$, i.e., $\MMM(B_1(X))$ with Hull-Kernel topology, where $\widehat{\mathscr M_f}= \{\widehat{M}\in \MMM(B_1(X)): f \in \widehat{M}\}$. \\
For any $f \in B_1(X)$ and $\widehat{M} \in \MMM(B_1(X))$ $f \in \widehat{M} \iff Z(f) \in Z[\widehat{M}] \iff Z(f) \in \mathscr U^p \iff p \in cl_{\widetilde{X}}Z(f)$, where $\Phi(\widehat{M})=p$. Hence $\Phi(\widehat{\mathscr M_f})=cl_{\widetilde{X}}Z(f)=\overline{Z(f)}$, for any $f \in B_1(X)$. Clearly, $\Phi$ exchanges the basic closed sets between $\MMM(B_1(X))$ and $\widetilde{X}$. Therefore, $\Phi$ is a homeomorphism between the structure space of $B_1(X)$ and $\widetilde{X}$ with Stone-topology.
\end{proof}
\begin{corollary}
$\widetilde{X}$ with Stone-topology is a compact Hausdorff space. 
\end{corollary}
\begin{proof}
Immediate, as $\MMM(B_1(X))$ is a compact Hausdorff space.
\end{proof}
\noindent The following result describes the collection of all maximal ideals of a Tychonoff space $X$. 
\begin{theorem}
A complete description of maximal ideals of the ring $B_1(X)$ is given by $\{\widehat{M^p}: p \in \widetilde{X}\}$, where $\widehat{M^p}= \{f \in B_1(X): p \in cl_{\widetilde{X}}Z(f)\}$. Further, if $p \neq q$ in $\widetilde{X}$ then $\widehat{M^p} \neq \widehat{M^q}$.
\end{theorem}
\begin{proof}
A $Z_B$-ultrafilter $\mathscr U^p$ on $X$ corresponds to a unique point $p$ in $\widetilde{X}$ and for all $ Z \in Z(B_1(X))$, $Z \in \mathscr U^p$ if and only if $p \in \overline{Z}$. i.e., $p \in cl_{\widetilde{X}}Z(f)$.\\
Since $\{\mathscr U^p : p \in \widetilde{X}\}$ is the collection of all $Z_B$-ultrafilters on $X$, it follows that $\{Z_B^{-1}[\mathscr U^p]: p \in \widetilde{X}\}$ is the collection of all maximal ideals of $B_1(X)$. Let $Z_B^{-1}[\mathscr U^p]= \widehat{M^p}$. Then $\widehat{M^p}=\{f \in B_1(X): Z(f) \in \mathscr U^p\}$ $=\{f \in B_1(X):p \in cl_{\widetilde{X}}Z(f)\}$.\\
Again if  $p \neq q$ in $\widetilde{X}$ then $\mathscr U^p \neq \mathscr U^q$, which implies $Z_B^{-1}[\mathscr U^p] \neq Z_B^{-1}[\mathscr U^q]$ and so, $\widehat{M^p} \neq \widehat{M^q}$.
\end{proof}

 \begin{theorem}
 	$\widehat{M^p}$ is a fixed maximal ideal in $B_1(X)$ if and only if $p \in X$.
 \end{theorem}
\begin{proof}
Let $p \in X$. Then $\widehat{M^p}=\{f \in B_1(X)\ : \ p \in cl_{\widetilde{X}}Z(f)\}=\{f \in B_1(X) \ : \ p \in \overline{Z(f)}\}$. We know $\overline{Z(f)}= \{p \in \widetilde{X}:Z(f) \in \mathscr U^p\}$. So, $p \in X \cap \overline{Z(f)} \implies p \in Z(f) \implies f(p)=0$. i.e., $\widehat{M^p}= \{f \in B_1(X): f(p)=0\}= \widehat{M_p}=$ a fixed maximal ideal. Conversely, $\widehat{M^q}$ is a fixed maximal ideal for some $q \in \widetilde{X}$. Since the collection of all fixed maximal ideals in the ring $B_1(X)$ is $\{\widehat{M_p}: p \in X\}$, where $\widehat{M_p}=\{f \in B_1(X)\ : f(p)=0\}$, we get $\widehat{M^q}=\widehat{M_p}$, for some $p \in X$. Hence, $\widehat{M^q}=\widehat{M_p}= \widehat{M^p}$ which implies $q =p \in X$.
\end{proof}
\section{Is $\MMM({B_1(X)})$ a compactification of $X$?}
\noindent As proposed in the introduction of this paper, we introduce two weaker forms of embedding which we call $F_\sigma$-embedding and weak $F_\sigma$-embedding of $X$ in $\MMM(B_1(X))$. Before we define such weak embeddings, we recall a result from ~\cite{LV} :
\begin{theorem}\cite{LV} \label{P1 thm_4.1} \textnormal{(i)} For any topological space $X$ and any metric space $Y$, $B_1(X,Y)$  $\subseteq \mathscr{F_\sigma}(X,Y)$, where $B_1(X,Y)$ denotes the collection of Baire one functions from $X$ to $Y$ and $\mathscr{F_\sigma}(X,Y)=\{f:X\rightarrow Y : f^{-1}(G)$ is an $F_\sigma$ set, for any open set $G \subseteq Y$\}.\\
 \textnormal{(ii)} For a normal topological space $X$, $B_1(X,\mathbb{R})$  $= \mathscr{F_\sigma}(X,\mathbb{R})$.
 \end{theorem}

\begin{definition}
A function $f : X \rightarrow Y$ is called \\
(i) \textbf{$F_\sigma$-continuous} if for any open set $U$ of $Y$ $f^{-1}(U)$ is an $F_\sigma$ set in $X$.\\
(ii) \textbf{weak $F_\sigma$ continuous} if for any basic open set $U$ of $Y$ $f^{-1}(U)$ is an $F_\sigma$ set in $X$.\\
(iii) \textbf{$F_\sigma$-embedding} if $f$ is injective, $F_\sigma$ continuous and $f^{-1} : f(X) \rightarrow X$ is continuous.\\
(iv) \textbf{weak $F_\sigma$-embedding} if $f$ is injective, weak $F_\sigma$ continuous and  $f^{-1} : f(X) \rightarrow X$ is continuous.
\end{definition}
\noindent It is quite easy to observe that a function $f : X \rightarrow Y$ is \textbf{$F_\sigma$-continuous} (respectively, \textbf{weak $F_\sigma$-continuous}) if and only if for any closed set (respectively, basic closed set) $C$ of $Y$ $f^{-1}(C)$ is a $G_\delta$ set in $X$.
\begin{remark}
In general, $F_\sigma$-continuous function is always weak $F_\sigma$-continuous. If each open set of $Y$ is expressible as a countable union of basic open sets then weak $F_\sigma$-continuity coincides with $F_\sigma$-continuity of the function. Since every open set of $\RR$ is a countable union of disjoint open intervals, a function $f : X \rightarrow \RR$ is weak $F_\sigma$-continuous if and only if it is $F_\sigma$-continuous.    
\end{remark} 
\begin{theorem}\label{EMB}
A $T_4$ space $X$ is densely weak $F_\sigma$-embedded in $\MMM(B_1(X))$.
\end{theorem}
\begin{proof}
Define $\psi : X \ra \MMM(B_1(X))$ by $\psi(x) = \widehat{M}_x$, where $\widehat{M}_x = \{f \in B_1(X) \ : f(x) = 0\ \}$. Then $\psi$ is an injective function, as $x \neq y$ implies that $\widehat{M}_x \neq \widehat{M}_y$. Since $\{\widehat{\mathscr M_f} \ : \ f \in B_1(X)\}$ is a base for closed sets for the hull-kernel topology on $\MMM(B_1(X))$, $\psi^{-1}(\widehat{\mathscr M_f}) = \{x\in X \ : \ \widehat{M}_x \in \widehat{\mathscr M_f}\}$ = $Z(f)$ which is a $G_\delta$ set in $X$. Since $\psi$ pulls back all basic closed sets of $\MMM(B_1(X))$ to $G_\delta$ sets of $X$, $\psi$ is a weak $F_\sigma$-continuous function. \\\\
$\{Z(g) \ : \ g \in C(X)\}$ is a base for closed sets for the topology of $X$, $(\psi^{-1})^{-1}(Z(g)) = \{\psi(x) \ : \ x \in Z(g)\}$ $= \{\widehat{M}_x \ : \ g \in \widehat{M}_x\}$ $= \widehat{\mathscr M_g} \cap \psi(X)$, a closed set in $\psi(X)$. Hence $\psi^{-1} : \psi(X) \ra X$ is a continuous function. 
Therefore, $X$ is weak $F_\sigma$-embedded in $\MMM(B_1(X))$. \\\\
That $\psi(X)$ is dense in $\MMM(B_1(X))$ follows from the next observation :
\begin{eqnarray*}
\overline{\psi(X)} &=& \{\widehat{M} \in B_1(X) \ :  \ \widehat{M} \supseteq \bigcap_{x\in X}\widehat{M}_x\}\\
				   &=& \{\widehat{M} \in B_1(X) \ :  \ \widehat{M} \supseteq \{0\}\}\\
				   &=& \MMM(B_1(X)).
\end{eqnarray*}
\end{proof}
\begin{corollary}
If every closed set of $\MMM(B_1(X))$ is expressible as a countable intersection of $\{\widehat{\mathscr M_f} \ : \ f \in B_1(X)\}$ then the $T_4$-space $X$ is densely $F_\sigma$-embedded in $\MMM(B_1(X))$.
\end{corollary}
\noindent We shall show that every $F_\sigma$-continuous function from a $T_4$ space $X$ to a compact Hausdorff space $Y$ has a unique continuous extension on $\MMM(B_1(X))$. To establish our claim, we need a lemma:
\begin{lemma}\label{Lem}
Let $X$ be a normal space and $\psi : X \ra Y$ (where $Y$ is any topological space) be a $F_\sigma$-continuous function. For each $h \in C(Y)$, $h\circ \psi \in B_1(X)$.
\end{lemma}
\begin{proof}
Let $C$ be a closed set in $\RR$. By continuity of $h$, $h^{-1}(C)$ is closed in $Y$. $F_\sigma$-continuity of $\psi$ implies that $\psi^{-1}(h^{-1}(C))$ is $G_\delta$. Hence, $(h\circ \psi)^{-1}(C)$ is a $G_\delta$-set in $X$. i.e., $h \circ \psi \in B_1(X)$. 
\end{proof}
\begin{theorem}\label{EXT}
If $f : X \rightarrow Y$ is a $F_\sigma$-continuous function from a $T_4$ space $X$ to a compact Hausdorff space $Y$ then there exists a unique continuous function $\widehat{f} : \MMM(B_1(X)) \ra Y$ such that $\widehat{f} \circ \psi = f$.  
\end{theorem}
\begin{proof}
Let $\widehat{M} \in \MMM(B_1(X))$. Construct $M_1 = \{g\in C(Y): g \circ f \in \widehat{M} \}$. By Lemma~\ref{Lem}, as $g\in C(Y)$ and $f$ is $F_\sigma$ continuous, $g\circ f \in B_1(X)$. Now it is easy to check that $M_1$ is a prime ideal in $C(Y)$ and therefore can be extended to a unique maximal ideal $M$ in $C(Y)$ (because, $C(Y)$ is a Gelfand ring). As $Y$ is compact, $M = M_y$, for some $y \in Y$. Hence, $g(y) = 0$, for all $g \in M_1$. In other words, $y \in \bigcap\limits_{g \in M_1} Z(g)$. It is also easy to observe that $\bigcap\limits_{g \in M_1} Z(g) = \{y\}$. So, define $\widehat{f} : \MMM(B_1(X)) \ra Y$ by 
$
\widehat{f}(\widehat{M}) = y
$, where $\bigcap\limits_{g \in M_1} Z(g)= \{y\}$.\\
It is also clear that $\widehat{f} \circ \psi = f$. \\\\
\textbf{$\widehat{f}$ is continuous :}  
Let $\widehat{M} \in \MMM(B_1(X))$. To check the continuity of $\widehat{f}$ at $\widehat{M}$. Let $W$ be any neighbourhood of $\widehat{f}\big( \widehat{M}\big)$ in $Y$. Since $Y$ is compact $T_2$ space, it is Tychonoff and hence $W$ contains a zero set neighbourhood of $\widehat{f}\big( \widehat{M}\big)$.\\
 Let $\widehat{f}\big( \widehat{M}\big) \in Y \smallsetminus Z(g_1) \subseteq Z(g_2) \subseteq W$, for some $g_1,g_2 \in C(Y)$.\\
 $\widehat{f}\big( \widehat{M}\big) \notin Z(g_1) \implies g_1 \notin M_1 \implies g_1 \circ f \notin \widehat{M}$. So, $\widehat{M} \in \big( \MMM(B_1(X)) \big) \smallsetminus \widehat{\mathscr M _{g_1 \circ f}}$ is a basic open set of $\MMM(B_1(X))$ containing $\widehat{M}$.\\
 Our claim is $\widehat{f}\bigg(\MMM(B_1(X)) \smallsetminus \widehat{\mathscr M _{g_1 \circ f}} \bigg) \subseteq W$.\\
Let $\widehat{N}\in \MMM(B_1(X)) \smallsetminus \widehat{\mathscr M _{g_1 \circ f}} \implies \widehat{N} \notin \widehat{\mathscr M _{g_1 \circ f}} \implies g_1 \circ f \notin \widehat{N} \implies g_1 \notin N_1$. Also $g_1.g_2=0 \in N_1$. So, $g_2 \in N_1$, (Since $N_1$ is a prime ideal), which implies $g_2 \circ f \in \widehat{N}$. Hence $g_2\big( \widehat{f}(\widehat{N})\big)=0$, i.e., $\widehat{f}\big( \widehat{N}\big) \subseteq Z(g_2) \subseteq W$. So, $\widehat{f}\bigg(\MMM(B_1(X)) \smallsetminus \widehat{\mathscr M _{g_1 \circ f}} \bigg) \subseteq W$.\\
The uniqueness of $\widehat{f}$ follows from the fact that, $\widehat{f}$ is continuous and $\psi(X)$ is dense in $\MMM(B_1(X))$.
 \end{proof}

\begin{theorem}\label{ctype}
For a $T_4$ space $X$, $B_1^*(X)$ is isomorphic to $C(\MMM(B_1(X)))$. In other words, for a $T_4$-space $X$, $B_1^*(X)$ is a C-type ring.
\end{theorem}
\begin{proof}
Let $f \in B_1^*(X)$. Since $X$ is normal and $f$ is Baire one, it is $F_\sigma$-continuous. Using Theorem~\ref{EXT}, $f$ has a continuous extension $\widehat{f} : \MMM (B_1(X)) \rightarrow [r, s]$ where $f(X) \subseteq [r, s]\subseteq \RR$.\\
Define $\eta : B_1^*(X) \rightarrow C(\MMM (B_1(X)))$ by $\eta(f) = \widehat{f}$. \\
We claim that $\eta$ is a ring isomorphism. Clearly, for any $f, g \in B_1^*(X)$ and for any fixed maximal ideal $\widehat{M}_x$ of $\MMM (B_1(X))$,
$$
(\widehat{f + g})(\widehat{M}_x) = \widehat{f}(\widehat{M}_x) + \widehat{g}(\widehat{M}_x)
$$
and
$$
(\widehat{fg})(\widehat{M}_x) = \widehat{f}\widehat{g}(\widehat{M}_x)
$$
Using the notation of Theorem~\ref{EMB}, $\psi(X) = \{\widehat{M}_x \ : \ x\in X\}$ is dense in $\MMM(B_1(X))$ and hence, $\widehat{f + g}= \widehat{f}+ \widehat{g}$ and $\widehat{fg}= \widehat{f}\widehat{g}$. 
So, $\eta (f + g) = \eta(f) + \eta(g)$ and $\eta(fg) = \eta(f) \eta(g)$. i.e., $\eta$ is a ring homomorphism. That $\eta$ is one-one is clear from the definition of the function $\eta$. We finally show that $\eta$ is surjective.\\
For each $h \in C(\MMM(B_1(X)))$, $h\circ \psi \in B_1(X)$ and $(h\circ \psi)(X) \subseteq [a, b]$, for some $a< b$ in $\RR$. i.e., $h\circ \psi \in B_1^*(X)$. Now, $\eta(h\circ \psi) = \widehat{h\circ\psi}$ where $(\widehat{h\circ\psi})(\widehat{M}_x) = (h\circ\psi)(x) = h(\widehat{M}_x)$. Since $\widehat{h\circ\psi} = h$ on $\psi(X)$ and $\psi(X)$ is dense in $\MMM (B_1(X))$, we conclude that $\eta(h\circ \psi) = h$.
\end{proof}
\begin{corollary}
If $X$ is a $T_4$ space then $\MMM(B_1(X)) \cong \MMM(B_1^*(X))$.
\end{corollary}
\begin{proof}
If two rings are isomorphic then their structure spaces are homeomorphic. So, $\MMM(B_1^*(X)) \cong \MMM(C(\MMM(B_1(X))))$. Since $\MMM(B_1(X))$ is compact, 
$$\MMM(C(\MMM(B_1(X)))) \cong \MMM(B_1(X)).$$ Hence $\MMM(B_1(X)) \cong \MMM(B_1^*(X))$.
\end{proof}
So far we have seen that for a normal space $(X, \tau)$, we get a compact Hausdorff space $\MMM(B_1(X))$ such that $(X, \tau)$ is densely weak $F_\sigma$-embedded in $\MMM(B_1(X))$ and every Baire one function on $X$ has a unique continuous extension on it. For which class of spaces we may expect $\MMM(B_1(X))$ to be the $Stone$-$\check{C}ech$ compactification of $(X,\tau)$? In what follows, we could partially resolve this matter.

Let $(X, \tau)$ be a $T_4$ space. Consider the collection $\BBB = \{Z(f) \ : \ f \in B_1(X)\}$. It is clear that $\emptyset = Z(1) \in \BBB$ and $Z(f) \cup Z(g) = Z(fg) \in \BBB$, for any $f, g \in B_1(X)$. So, $\BBB$ forms a base for closed sets for some topology $\sigma$ on $X$. Certainly, $\tau \subseteq \sigma$, as $\{Z(f) \ : \ f \in C(X)\}$ is a base for closed sets for the topology $\tau$. \\\\
If $B_1(X) = C(X)$ then of course $\tau = \sigma$. Does the converse hold? i.e., if $\sigma = \tau$ then does it imply $B_1(X) = C(X)?$ Before answering this question, we observe that $\MMM(B_1(X))$ is indeed the $Stone$-$\check{C}ech$ compactification of $(X, \sigma)$.\\
In the following theorem, to avoid any ambiguity, we use the notation $\MMM (B_1(X, \tau))$ to denote $\MMM(B_1(X))$ and $\MMM(C(X, \sigma))$ to denote the structure space of $C(X)$, where $X$ has $\sigma$ topology on it. We also use the notation $\beta X_\sigma$ to denote the $Stone$-$\check{C}ech$ compactification of $(X, \sigma)$.  
\begin{theorem}\label{SC}
For a $T_4$ space $(X,\tau)$, $\MMM(B_1(X, \tau))$ is the $Stone$-$\check{C}ech$ compactification of $(X, \sigma)$. 
\end{theorem}
\begin{proof}
We first observe that $\psi^* : (X, \sigma) \rightarrow \MMM(B_1(X, \tau))$ given by $x \mapsto \widehat{M}_x$ is an embedding. That $\psi^*$ is a one-one map is already proved in Theorem~\ref{EMB}. It is easy to observe that $\psi^*$ exchanges base for closed sets of $(X, \sigma)$ and $\MMM(B_1(X, \tau))$: For any $f\in B_1(X, \tau)$,
\begin{eqnarray*}
\psi^*(Z(f)) &=& \{\psi^*(x) \ : \ f(x) = 0\}\\
		   &=& \{\widehat{M}_x \ : \ f(x) = 0\}\\
		   &=& \{\widehat{M}_x \ : \ f\in \widehat{M}_x\}\\
		   &=& \widehat{\mathscr M}_f \cap \psi^*(X)
\end{eqnarray*}
Therefore, $(X, \sigma)$ is densely embedded in the compact Hausdorff space $\MMM(B_1(X, \tau))$. We now show that for any continuous function $f : (X, \sigma) \rightarrow Y$  (where $Y$ is any compact Hausdorff space), there exists $\widehat{f} : \MMM(B_1(X, \tau)) \rightarrow Y$ such that $\widehat{f}\circ \psi^* = f$.\\
Let $\widehat{M} \in \MMM(B_1(X))$. Construct $M_1 = \{g\in C(Y) \ : \ g\circ f \in \widehat{M}\}$. It is easy to check that $M_1$ is a prime ideal of $C(Y)$ and hence can be extended to a unique maximal ideal of $C(Y)$. $Y$ being a compact Hausdorff space, every maximal ideal of $C(Y)$ is fixed and hence, $M_1 \subseteq M_y$, for some $y\in Y$. So, for all $g \in M_1$, $g(y) = 0$ which implies that $y \in \bigcap\limits_{g\in M_1} Z(g)$. Clearly, $\bigcap\limits_{g\in M_1} Z(g) = \{y\}$. Define $\widehat{f} : \MMM(B_1(X, \tau)) \rightarrow Y$ by $\widehat{f}(\widehat{M}) = y$, whenever $\bigcap\limits_{g\in M_1} Z(g) = \{y\}$. Proceeding as in Theorem~\ref{EXT}, we observe that $\widehat{f}$ is a unique continuous function satisfying $\widehat{f}\circ \psi^* = f$. Hence, $\MMM(B_1(X, \tau))$ is the $Stone$-$\check{C}ech$ compactification of $(X, \sigma)$.  
\end{proof}
\begin{corollary}\label{COR_SC}
For any $T_4$ space $(X, \tau)$, $\MMM(B_1(X, \tau)) \cong \beta X_\sigma \cong \MMM(C(X, \sigma))$.
\end{corollary}
\begin{proof}
Follows from the theorem and the fact that the structure space of $C(X, \sigma)$ is $\beta X_\sigma$.
\end{proof}
The following gives a complete description of the maximal ideals of ${B_1}^*(X)$ :
\begin{theorem}
For a $T_4$ space $X$, $\{\widehat{M}^{* p} \ : \ p \in \beta X_\sigma\}$ is the complete collection of maximal ideals of ${B_1}^*(X)$, where $\widehat{M}^{* p} = \{f \in B_1^*(X) \ : \ \widehat{f}(p) = 0\}$. Also, $p \neq q$ implies $\widehat{M}^{* p} \neq \widehat{M}^{* q}$. Moreover, $\widehat{M}^{*p}$ is a fixed maximal ideal if and only if $p \in X$. 
\end{theorem}
\begin{proof}
By Theorem~\ref{ctype}, the map $\eta : f \rightarrow \widehat{f}$ is an isomorphism from $B_1^*(X)$ onto $C(\mathcal{M}(B_1(X)))$. So, there is a one-one correspondence between the maximal ideals of $B_1^*(X)$ and those of $C(\mathcal{M}(B_1(X)))$. $\mathcal{M}(B_1(X))$ being compact, every maximal ideal of the ring $C(\mathcal{M}(B_1(X)))$ is of the form $\{h \in C(\mathcal{M}(B_1(X))) \ : \ h(p) = 0\}$, where $p \in \mathcal{M}(B_1(X)) \cong \beta X_\sigma$. So, the maximal ideals of $B_1^*(X)$ are given by \\
$\eta^{-1}\left( \{h \in C(\mathcal{M}(B_1(X))) \ : \ h(p) = 0\}\right) = \{f \in {B_1}^*(X) \ : \ \widehat{f}(p) = 0\} = \widehat{M}^{*p}$ (say),
for each $p \in \beta X_\sigma$. \\
$p \neq q$ $\Rightarrow$ $\{h \in C(\mathcal{M}(B_1(X))) \ : \ h(p) = 0\} \neq $ $\{h \in C(\mathcal{M}(B_1(X))) \ : \ h(q) = 0\}$ and so, $\eta^{-1}\left(\{h \in C(\mathcal{M}(B_1(X))) \ : \ h(p) = 0\}\right) \neq $ $\eta^{-1}\left(\{h \in C(\mathcal{M}(B_1(X))) \ : \ h(q) = 0\}\right)$. i.e., $\widehat{M}^{* p} \neq \widehat{M}^{* q}$. \\
If $p \in X$ then clearly, $\widehat{M}^{* p} = \{f \in {B_1}^*(X) \ : \ f(p) = 0\} = \widehat{M}^*_p$, the fixed maximal ideal of ${B_1}^*(X)$. \\
If $q \in \beta X_\sigma \setminus X$, then we claim that $\widehat{M}^{* q}$ is not fixed. If possible, it is a fixed maximal ideal of ${B_1}^*(X)$. Then $\widehat{M}^{* q} = {\widehat{M}^ *}_p$ for some $p \in X$. But in that case, $\widehat{M}^{* q} = \widehat{M}^{* p}$ and consequently, $p = q$.
\end{proof}
It is well known by Gelfand Kolmogoroff Theorem that $\{M^p \ : \ p\in \beta X_\sigma\}$ is precisely the collection of all maximal ideals of $C(X, \sigma)$. So, $\{\widehat{M}^p \ : \ p\in \beta X_\sigma\}$ is the exact collection of all maximal ideals of $B_1(X)$, where $X$ has $\tau$ topology on it. Moreover, for each $p\in X$, $\widehat{M}^p = \widehat{M}_p$, a fixed maximal ideal of $B_1(X)$. So, under the isomorphism of Corollary~\ref{COR_SC}, $\widehat{M}_p$ of $\MMM(B_1(X, \tau))$ corresponds to $M_p$ of $\MMM(C(X, \sigma))$. As a result, we get the following:
\begin{theorem}\label{COMP_FIX}
$(X, \sigma)$ is compact if and only if each maximal ideal of $B_1(X, \tau)$ is fixed.
\end{theorem}
\begin{proof}
Follows from Corollary~\ref{COR_SC}.
\end{proof}
\noindent It is easy to check that in $B_1(X)$, every maximal ideal is fixed if and only if every ideal is fixed.  \\\\
Also, it is evident from the last theorem that for a $T_4$ space $(X, \tau)$, if each maximal ideal of $B_1(X, \tau)$ is fixed then $(X, \sigma)$ is compact and therefore $(X, \tau)$ is also compact. But $(X, \tau)$ being Hausdorff it is maximal compact and hence, $\tau = \sigma$. On the other hand, if $(X, \tau)$ is compact then it is maximal compact (since it is Hausdorff) and therefore, $\sigma$ is never compact if $\tau \neq \sigma$. As a consequence of this we get the following:
\begin{theorem}
If $(X, \tau)$ is a compact Hausdorff space and $\sigma$ is strictly finer than $\tau$ then there exists at least one free (maximal) ideal in $B_1(X, \tau)$. 
\end{theorem} 
\begin{proof}
Follows from Theorem~\ref{COMP_FIX} and the fact that $(X, \tau)$ becomes maximal Hausdorff.
\end{proof}
It is now a relevant query, for which spaces $\sigma$ remains strictly finer than $\tau$. We claim that for a $T_4$ space $(X, \tau)$, $\sigma = \tau$ if and only if $B_1(X) = C(X)$ and we establish our claim in what follows next. 
\begin{theorem}\label{sigma_tau}
Let $(X, \tau)$ be a $T_4$ space and $\{Z(f) \ : \ f \in B_1(X, \tau)\}$ be a base for closed sets for the topology $\sigma$ on $X$. Then $\sigma = \tau$ if and only if $B_1(X,\tau) = C(X,\tau)$. 
\end{theorem}
\begin{proof}
If $B_1(X,\tau) = C(X,\tau)$ then certainly $\sigma = \tau$. For the converse, let $\sigma = \tau$. Consider any $f \in B_1(X,\tau)$ and any basic open set $(a, b)$ of $\RR$. Then $f^{-1}(a, b) = \{x\in X \ : \ a < f(x) < b\} = X \setminus \left(\{x\in X \ : f(x) \leq a\} \cup \{x\in X \ : \ f(x) \geq b\}\right)$ $= (X \setminus Z(g_1)) \cap (X \setminus Z(g_2)) = U$ (say), where $g_1, g_2 \in B_1(X,\tau)$. Since $Z(g_1)$ and $Z(g_2)$ are closed in $(X, \sigma)$, $U$ is open in $(X, \sigma)$. Hence, $f \in C(X, \sigma)$. By hypothesis, $\sigma = \tau$ and therefore, $C(X, \tau) \subseteq B_1(X, \tau) \subseteq C(X, \sigma) = C(X, \tau)$. i.e., $B_1(X, \tau) = C(X, \tau)$. 
\end{proof}
Consequently, from Theorem~\ref{COMP_FIX} and Theorem~\ref{sigma_tau}, it follows that for a $T_4$ space $(X, \tau)$, if a non-continuous Baire one function exists then $B_1(X)$ has a free (maximal) ideal.\\
In general, $B_1(X) = C(X)$ does not always imply that $X$ is discrete. For example, if $X$ is a P-space then $B_1(X) = C(X)$ ~\cite{MW}. We shall now show that for a particular class of topological spaces, e.g., for perfectly normal $T_1$ spaces, $B_1(X) = C(X)$ is equivalent to the discreteness of the space. 
\begin{theorem}\label{SigmaIsDisc}
If $(X, \tau)$ is a perfectly normal $T_1$-space then $\sigma$ is the discrete topology on $X$. 
\end{theorem}
\begin{proof}
Let $\{y\}$ be any singleton set in $(X, \sigma)$. Since $T_1$-ness is an expansive property and $(X, \tau)$ is $T_1$, it follows that $(X, \sigma)$ is also $T_1$. So, $\{y\}$ is closed in $(X, \sigma)$. Since $(X, \tau)$ is perfectly normal, $\{y\} = Z(g)$, for some $g \in C(X, \tau)$. Define $\chi_y : X \ra \RR$ as follows :
$$
\chi_y(x) = \begin{cases}
			1 & \ x=y\\
			0 & otherwise.
			\end{cases}
$$
It is easy to check that $\chi_y \in B_1(X,\tau)$. Also $Z(\chi_y) = X \setminus \{y\}$ which is an open set in $(X, \tau)$ and hence open in $(X, \sigma)$. By definition of $(X, \sigma)$, $Z(\chi_y)$ is a closed set in $(X, \sigma)$. Hence, $\{y\}$ is both open and closed in $(X, \sigma)$ which shows by arbitrariness of $\{y\}$, $\sigma$ is the discrete topology on $X$.
\end{proof}
 
\begin{theorem}\label{NotSC}
For a perfectly normal $T_1$ space $(X, \tau)$ the following statements are equivalent:
\begin{enumerate}
\item $(X, \tau)$ is discrete.
\item $B_1(X) = C(X)$.
\item ${B_1}^*(X) = C^*(X)$.
\item $\MMM({B_1}^*(X))$ is the $Stone$-$\check{C}ech$ compactification of $(X, \tau)$.
\item $\MMM(B_1(X))$ is the $Stone$-$\check{C}ech$ compactification of $(X, \tau)$.
\end{enumerate}
\end{theorem}
\begin{proof}
$(1) \Rightarrow (2)$ : Immediate.\\\\
$(2) \Rightarrow (3)$ : $C^*(X) = {B_1}^*(X) \cap C(X) =  {B_1}^*(X) \cap B_1(X) = {B_1}^*(X)$. \\\\
$(3) \Rightarrow (4)$ : ${B_1}^*(X) = C^*(X)$ implies that $\MMM({B_1}^*(X)) = \MMM(C^*(X))$ and it is well known that $\MMM(C^*(X))$ is the $Stone$-$\check{C}ech$ compactification of $(X, \tau)$. \\\\
$(4) \Rightarrow (5)$ : Follows from the fact that $\MMM({B_1}^*(X)) \cong \MMM(B_1(X))$.\\\\
$(5) \Rightarrow (1)$ : By Theorem~\ref{SC}, $\MMM(B_1(X, \tau))$ is the $Stone$-$\check{C}ech$ compactification of $(X, \sigma)$. By (3),  $\MMM(B_1(X, \tau))$ is the $Stone$-$\check{C}ech$ compactification of $(X, \tau)$. Then the identity map $I : \MMM(B_1(X, \tau)) \ra \MMM(B_1(X, \tau))$ when restricted on $(X, \tau)$ becomes a homeomorphism between $(X, \tau)$ and $(X, \sigma)$. Hence, $\sigma = \tau$. Using Theorem~\ref{SigmaIsDisc}, $(X, \tau)$ is a discrete space.   
\end{proof}
\begin{remark}
There are plenty of non-discrete perfectly normal $T_1$ spaces (for example, any non-discrete metric space; in particular, $\RR$). By Theorem~\ref{NotSC}, we may certainly assert that at least for those spaces $\MMM(B_1(X))$ is never its $Stone$-$\check{C}ech$ compactification. Theorem~\ref{NotSC} also assures the existence of non-continuous Baire one functions in any non-discrete perfectly normal $T_1$ space.  
\end{remark}

\end{document}